\documentclass[a4paper]{amsart}
\usepackage{amsmath,amssymb,amsthm, url, color, lscape}
\usepackage[pdftex]{graphicx}

\title{Generic level $p$ Eisenstein congruences for $\text{GSp}_4$.}
\author{Dan Fretwell}
\thanks{Dan Fretwell, Heilbronn Institute for Mathematical Research, School of Mathematics, University of Bristol, U.K. Email: \url{daniel.fretwell@bristol.ac.uk}\\  $2010$ Mathematics Subject Classification: $11$F$33$, $11$F$70$, $11$F$80$.}
\date{}

\theoremstyle{plain}
\newtheorem{thm}{Theorem}[section]
\newtheorem{lem}[thm]{Lemma}
\newtheorem{cor}[thm]{Corollary}
\newtheorem{prop}[thm]{Proposition}
\newtheorem{conj}[thm]{Conjecture}

\theoremstyle{definition}

\begin{document}
\maketitle

\begin{center}
\textbf{Abstract}
\end{center}

We investigate level $p$ Eisenstein congruences for $\text{GSp}_4$, generalisations of level $1$ congruences predicted by Harder. By studying the associated Galois and automorphic representations we see conditions that guarantee the existence of a paramodular form satisfying the congruence. This provides theoretical justification for computational evidence found in the author's previous paper.

\section{Introduction}

Consider the discriminant function $\Delta(z) = q\prod_{n=1}^{\infty}(1-q^n)^{24}$ (where $q = e^{2\pi i z}$). A famous observation, due to Ramanujan, is that the Fourier coefficients $\tau(n)$ of $\Delta$ satisfy the congruence: \[\tau(n) \equiv \sigma_{11}(n) \bmod 691.\] A natural way to view this is as a congruence between the Hecke eigenvalues of the unique normalised weight 12 cusp form $\Delta$ and the weight 12 Eisenstein series $E_{12}$. The modulus $691$ appears since it divides the numerator of $\frac{\zeta(12)}{\pi^{12}}$, a quantity which appears in the constant term of $E_{12}$.

Since the work of Ramanujan there have been many generalizations of his congruences. Indeed by looking for big enough primes dividing the numerator of $\frac{\zeta(k)}{\pi^k}$, i.e. $\frac{B_k}{2k}$, one can provide similar congruences at level $1$ between cusp forms and Eisenstein series of weight $k$ \cite{datskovsky}. In fact one can also give ``local origin" congruences between higher level cusp forms and level $1$ Eisenstein series by extending the divisibility criterion to include Euler factors of $\zeta(k)$ rather than the global values of $\zeta(s)$ (see \cite{danneil} for results and examples).

One can study Eisenstein congruences for genus $2$ Siegel modular forms. There are many ways to generalise. In this paper we consider a particular conjectural congruence for paramodular forms of level $p$, an extension of a congruence predicted at level $1$ by Harder \cite{harder1}.

Given $k'\geq 0$ and $N\geq 1$ let $S_{k'}(\Gamma_0(N))$ denote the space of weight $k'$ elliptic cusp forms for $\Gamma_0(N)$ and let $S_{k'}^{\text{new}}(\Gamma_0(N))$ denote the subspace of new forms. 

For $j,k\geq 0$ let $V_{j,k}$ denote the representation $\text{Symm}^{j}(\mathbb{C}^2)\otimes\text{det}^k$ of $\text{GL}_2(\mathbb{C})$. Then $S_{j,k}(K(N))$ will denote the space of genus $2$, $V_{j,k}$-valued Siegel cusp forms for the paramodular group: \[K(N) = \left\{\left(\begin{array}{cccc}* & N* & * & *\\ * & * & * & N^{-1}*\\ * & N* & * & *\\ N* & N* & N* & *\end{array}\right)\right\} \cap \text{Sp}_4(\mathbb{Q}),\] where the stars represent integers.

For a normalised eigenform $f\in S_{k}(\Gamma_0(N))$ let $\Lambda(f,s)$ denote its completed L-function. For each critical value $1\leq m \leq k-1$ there exists a Deligne period $\Omega_m$ such that $\Lambda_{alg}(f,m) = \frac{\Lambda(f,m)}{\Omega_m}\in \overline{\mathbb{Q}}$, well defined upto multiplication by $\mathcal{O}_{\mathbb{Q}_f}^{\times}$. In fact $\Omega_m$ only depends on the parity of $m$. 

With this choice of period it makes sense to talk about divisibility of critical values of the L-function. Harder's original congruence suggests that for $N=1$, large enough primes dividing the numerators of these $L$-values should give Eisenstein congruences. In this paper we will be interested in the following level $p$ version.

\begin{conj}(Level $p$ paramodular Harder's conjecture)\label{V}

Let $j > 0$ and $k\geq 3$ and let $f\in S_{j+2k-2}^{\text{new}}(\Gamma_0(p))$ be a normalized Hecke eigenform away from $p$ with eigenvalues $a_q\in\mathcal{O}_f$. Suppose that $\text{ord}_\lambda(\Lambda_{alg}(f,j+k)) > 0$ for some prime $\lambda$ of $\mathbb{Q}_f$ lying above a rational prime $l > j+2k-2$ (with $l\neq p$). 

Then there exists a Hecke eigenform $F\in S_{j,k}^{\text{new}}(K(p))$ away from $p$ with eigenvalues $b_q\in\mathcal{O}_F$ satisfying \[b_q \equiv q^{k-2} + a_q + q^{j+k-1} \bmod \Lambda\] for all primes $q\neq p$ (where $\Lambda$ is some prime lying above $\lambda$ in the compositum $\mathbb{Q}_f\mathbb{Q}_F$).
\end{conj}

The $j=0$ version of the above conjecture gives congruences between newforms for $K(p)$ and Saito-Kurokawa lifts of forms for $\Gamma_0(p)$. Such congruences have been studied in detail, for example in \cite{brown}.

It should be noted that Harder's level $1$ congruence remains unproved and evidence is rare. The same can be said about higher level generalisations. In the author's previous paper computational evidence was given for Conjecture \ref{V} at levels $p=2,3,5,7$ by using algebraic modular forms \cite{fretwell}. 

Naturally one asks why there is a preference of considering paramodular forms over say $\Gamma_0(p)$. The aim of this paper is to show that, given the existence of a ``generic" level $p$ congruence of the above type, it is likely that the genus $2$ form can be taken to be paramodular. 

More specifically we consider an automorphic representation $\pi_F = \otimes \pi_{F,q}$ attached to a Siegel modular form of genus $2$, weight $(j,k)$. Assuming $\pi_{F,q}$ is unramified for all $q\neq p$ and that $F$ satisfies the congruence in Conjecture \ref{V} the following result gives the required limitations on $\pi_{F,p}$.

\begin{thm}\label{PQ}
Let $e(\Lambda), f(\Lambda)$ be the ramification index and the inertia degree of the extension $K_{\Lambda}/\mathbb{Q}_l$. Let $k' = j+2k-2$.
\begin{enumerate}
\item{If $l\geq \text{max}\{6f(\Lambda)+2, e(\Lambda)+2\}$ then $\pi_{F,p}$ is induced from the Borel subgroup of $\text{GSp}_4(\mathbb{Q}_p)$.}
\item{If further we have $p^{j+2t-2}\not\equiv 1 \bmod \Lambda$ for $t=0,1,2,3$ then $\pi_{F,p}$ is of type I or II (see the Appendix for the classification).}
\item{If further $\text{ord}_{\Lambda}\left(\frac{B_{k'}(p^{k'}-1)}{2k'}\right) = 0$ then either $\pi_{F,p}$ is of type II$_a$ or there exists $g\in S_{k'}(\text{SL}_2(\mathbb{Z}))$ satisfying the congruence.}
\end{enumerate}
\end{thm}

Being of type II$_a$ guarantees the existence of a new $K(p)$ fixed vector, hence that we may find $F\in S_{j,k}^{\text{new}}(K(p))$ satisfying the congruence. 

If $\text{ord}_{\Lambda}\left(\frac{B_{k'}(p^{k'}-1)}{2k'}\right) > 0$ then it is known that $f$ will satisfy a simpler ramanujan congruence \cite{danneil}. Thus we see that the condition in part three of the above is controlling the existence of a ``simpler" congruence. The conditions in parts one and two are not very restrictive since $l$ is large in general relative to $p, e(\Lambda)$ and $f(\Lambda)$.

\section{Proving Theorem \ref{PQ}}
We wish to justify the use of paramodular forms in the statement of Conjecture \ref{V}. As discussed in the introduction we will do this by proving Theorem \ref{PQ}. In order to do this we will fix the following notation:

\begin{itemize}
\item{$\pi_{F} = \otimes_{q\leq \infty}\pi_{F,q}$ is an automorphic representation of GSp$_4$ attached to some Siegel modular form of weight (j,k). We will assume that $\pi_{F}$ is unramified away from $p$. The form $F$ is assumed to be a Hecke eigenform away from $p$ with eigenvalues $b_q\in\mathcal{O}_F$.}
\item{$f\in S_{j+2k-2}^{\text{new}}(\Gamma_0(p))$ is a normalized Hecke eigenform away from $p$, wth eigenvalues $a_q\in\mathcal{O}_f$. We write $k' = j+2k-2$. Attached to $f$ is an automorphic representation $\pi_{f} = \otimes_{q\leq \infty}\pi_{f,q}$ of GL$_2$.}
\item{$K = \mathbb{Q}_f\mathbb{Q}_F$ is the compositum of coefficient fields of $f$ and $F$.}
\item{$\Lambda$ is a prime of $K$ lying above a rational prime $l\neq p$ satisfying $l > j+2k-2 > 4$. Associated to $\Lambda$ is a completion $K_{\Lambda}$, valuation ring $\mathcal{O}_{\Lambda}$ and residue field $\mathbb{F}_{\Lambda}$.}
\item{$\rho_f$ is the $2$-dimensional $\Lambda$-adic Galois representation associated to $f$, realised over $\mathcal{O}_{\Lambda}$. The mod $\Lambda$ semisimple reduction of this is $\overline{\rho}_F$. Also for each prime $q$ we have the restriction $\rho_{f,q}$ to $\text{Gal}(\overline{\mathbb{Q}}_q/\mathbb{Q}_q)$.}
\item{$\rho_F$ is the $4$-dimensional $\Lambda$-adic Galois representation associated to $F$, also realised over $\mathcal{O}_{\Lambda}$. See \cite{weissauer} for details. Again we have a mod $\Lambda$ semisimple reduction $\overline{\rho}_F$ and restrictions $\rho_{F,q}$ to $\text{Gal}(\overline{\mathbb{Q}}_q/\mathbb{Q}_q)$.}
\end{itemize}

From now on we assume that the pair $(f,F)$ satisfy Conjecture \ref{V}, with $\Lambda$ being the modulus of the congruence. We will prove the parts of Theorem \ref{PQ} in reverse order.

We will need the following well known results about Galois representations attached to elliptic modular forms.

\begin{thm}{(Deligne)}\label{B} Let $f\in S_k(\Gamma_0(N))$ be a Hecke eigenform for all Hecke operators $q\nmid N$ with eigenvalues $a_q$. Let $l$ be a prime satisfying $2 \leq k \leq l+1$ and $a_l \not\equiv 0 \bmod l$. Then $\overline{\rho}_{f,l}$ is reducible and \[\overline{\rho}_{f,l} \sim \left(\begin{array}{cc}\overline{\chi}_l^{k-1}\overline{\lambda}_{a_l^{-1}} & \star\\ 0 & \overline{\lambda}_{a_l}\end{array}\right),\] where $\lambda_a$ is the unramified character $\text{Gal}(\overline{\mathbb{Q}}_l/\mathbb{Q}_l) \longrightarrow \overline{\mathbb{Z}}_l^{\times}$ such that $\lambda_a(\phi_l) = a$ (here $\phi_l$ is a frobenius element in $\text{Gal}(\overline{\mathbb{Q}_l}/\mathbb{Q}_l)$).
\end{thm}

\begin{thm}{(Fontaine)}\label{A} Suppose that $f$ and $l$ are as above but that $a_l \equiv 0 \bmod l$. Then $\overline{\rho}_{f,l}$ is irreducible.
\end{thm}

Naturally one asks about the structure of $\overline{\rho}_{f,p}$ for $p|N$. The following theorem can be found on p.$309$ of \cite{hida}.

\begin{thm}{(Langlands-Carayol)} \label{PUI}
Suppose $p$ is a prime such that $\text{ord}_p(N) = 1$. Then: \[\overline{\rho}_{f,p} \sim \left(\begin{array}{cc}\overline{\chi}_l \overline{\lambda}_{a} & \star\\ 0 & \overline{\lambda}_{a}\end{array}\right)\] for some fixed $a$.
\end{thm}

\subsection{Proving Theorem \ref{PQ}$(3)$}

Assume that $\pi_{F,p}$ is of type I or II. Our aim is to show that if $\text{ord}_{\Lambda}\left(\frac{B_{k'}(p^{k'}-1)}{2k'}\right) = 0$ then either $\pi_{F,p}$ is of type II$_a$ (so that $\pi_{F,p}$ has new $K(p)$-fixed vectors) or that there exists $g\in S_{k'}(\text{SL}_2(\mathbb{Z}))$ replacing $f$ in the congruence.

To do this we first we translate the congruence into a result about Galois representations. 

\begin{lem}
\[\overline{\rho}_F \sim \overline{\rho}_f \oplus \overline{\chi}_l^{k-2} \oplus \overline{\chi}_l^{j+k-1},\] where $\chi_l$ is the $l$-adic cyclotomic character. 
\end{lem}

\begin{proof}
By assumption we have for each $q \neq p$: \[b_q \equiv a_q + q^{k-2} + q^{j+k-1} \bmod \Lambda.\] In terms of mod $\Lambda$ representations this gives $\text{tr}(\overline{\rho}_F(\phi_q)) = \text{tr}((\overline{\rho_f}\oplus\overline{\chi}_l^{k-2}\oplus\overline{\chi}_l^{j+k-1})(\phi_q))$ for all $q\neq p,l$.

The Cebotarev density theorem gives \[\text{tr}(\overline{\rho}_F) = \text{tr}(\overline{\rho}_f\oplus\overline{\chi}_l^{k-2}\oplus\overline{\chi}_l^{j+k-1}).\] Then since $l>4$ the result follows by the Brauer-Nesbitt theorem.
\end{proof}  

It will be handy to know when $\bar{\rho}_f$ is irreducible. The Bernoulli criterion forces this.

\begin{lem}
If $\overline{\rho}_f$ is reducible then ord$_{\Lambda}\left(\frac{B_{k'}(p^{k'}-1)}{2k'}\right)>0$.
\end{lem}

\begin{proof}
Suppose $\overline{\rho}_f$ is reducible. Then after a suitable choice of basis: \[\overline{\rho}_f = \left(\begin{array}{cc}\alpha & \star\\ 0 & \beta\end{array}\right),\] where $\alpha, \beta$ are two characters Gal$(\overline{\mathbb{Q}}/\mathbb{Q}) \rightarrow \mathbb{F}_{\Lambda}^{\times}$. Notice that the image of these characters is abelian.

Now it is known that $\overline{\rho}_f$ is unramified at all primes $q\nmid pl$ and so $\alpha$ and $\beta$ must be unramified at the same primes. This forces $\alpha = \overline{\chi}_l^m\epsilon_1$ and $\beta = \overline{\chi}_l^n\epsilon_2$ where $\epsilon_1,\epsilon_2$ are unramified outside $p$. 

To see this let $\chi: \text{Gal}(\overline{\mathbb{Q}}/\mathbb{Q})\rightarrow\mathbb{F}_{\Lambda}^{\times}$ be a character unramified at all $q\nmid pl$. Note that by global class field theory $\alpha$ and $\beta$ factor through $\text{Gal}(\mathbb{Q}(\mu_{p^{\infty}},\mu_{l^{\infty}})/\mathbb{Q})$ (where $\mu_{p^{\infty}}$ denotes the set of $p$th power roots of unity, similarly for $l$). The field $\mathbb{Q}(\mu_{p^{\infty}},\mu_{l^{\infty}})$ is the maximal abelian extension of $\mathbb{Q}$ unramified outside $pl$. 

We find that $\text{Gal}(\mathbb{Q}(\mu_{p^{\infty}},\mu_{l^{\infty}})/\mathbb{Q}) \cong \text{Gal}(\mathbb{Q}(\mu_{p^{\infty}})/\mathbb{Q})\times\text{Gal}(\mathbb{Q}(\mu_{l^{\infty}})/\mathbb{Q})$ since $p$ and $l$ are coprime. Hence $\chi=\delta\epsilon$ where $\delta$ is unramified outside of $l$ and $\epsilon$ is unramified outside of $p$. 

To prove the claim that $\delta$ is a power of $\overline{\chi}_l$ note that $\text{Gal}(\mathbb{Q}(\mu_{l^{\infty}})/\mathbb{Q}) \cong \mathbb{Z}_l^{\times} \cong \left(\mathbb{Z}/(l-1)\mathbb{Z}\right) \times \mathbb{Z}_l$ (using the fact that $l>2$). By continuity of Galois representations we know that $\delta$ has to be trivial on $l^t\mathbb{Z}_l$ for some $t\geq 0$ and so $\delta$ induces a representation of $\left(\mathbb{Z}/(l-1)\mathbb{Z}\right)\times\left(\mathbb{Z}/l^t\mathbb{Z}\right)$. But since $l$ is coprime to $|\mathbb{F}_{\Lambda}^{\times}| = N(\Lambda)-1$ the image of the second component must be trivial. The characters of $\left(\mathbb{Z}/(l-1)\mathbb{Z}\right)\cong \left(\mathbb{Z}/l\mathbb{Z}\right)^{\times}$ are exactly the powers of $\overline{\chi}_l$. Thus $\chi = \overline{\chi}_l^s\epsilon$ for some integer $s$.

Continuing we now see that since det$(\overline{\rho}_f(\phi_q)) \equiv q^{k'-1} \bmod \Lambda$ for all $q\nmid pl$ it must be that $\epsilon_2 = \epsilon_1^{-1}$.

Thus: \[\overline{\rho}_f = \left(\begin{array}{cc}\overline{\chi}_l^m \epsilon & \star\\ 0 & \overline{\chi}_l^n\epsilon^{-1}\end{array}\right).\] A comparison of Artin conductors (p.$39$ of \cite{weise}) shows that $\epsilon$ is trivial. Indeed the Artin conductor of $\overline{\rho}_f$ is known to be $p$ whereas if $\epsilon$ is non-trivial then the Artin conductor would be at least $p^2 > p$.

Now recall $4 < k' < l$. Also it must be the case that $a_l \not\equiv 0 \bmod \Lambda$ (otherwise $\overline{\rho}_{f,l}$ is irreducible by Theorem \ref{A}, contradicting the reducibility of $\overline{\rho}_f$). 

Thus by Theorem \ref{B} we see that $\overline{\rho}_{f,l}$ must possess an unramified composition factor, hence one of $\overline{\chi}_l^m, \overline{\chi}_l^n$ must be unramified at $l$. Since all non-trivial powers of $\overline{\chi}_l$ are ramified at $l$ this means one of the composition factors is trivial. It is then clear that the other composition factor must be $\overline{\chi}_l^{k'-1}$.

Hence: \[\overline{\rho}_f = \left(\begin{array}{cc}1 & \star\\ 0 & \overline{\chi}_l^{k'-1}\end{array}\right)\quad\text{or}\quad \left(\begin{array}{cc}\overline{\chi}_l^{k'-1} & \star\\ 0 & 1\end{array}\right).\]

In either case comparing traces of Frobenius at $q\neq p,l$ gives the Ramanujan congruence: \[a_q \equiv 1 + q^{k'-1} \bmod \Lambda.\] By Proposition $4.2$ of \cite{danneil} it must then be that ord$_{\Lambda}\left(\frac{B_{k'}(p^{k'}-1)}{2k'}\right) > 0$.
\end{proof}

\begin{prop}
Suppose $\pi_{F,p}$ is of type I or II and that ord$_{\Lambda}\left(\frac{B_{k'}(p^{k'}-1)}{2k'}\right) = 0$. Then either $\pi_{F,p}$ is of type II$_a$ or there exists a level one normalized newform $g\in S_{k'}(\text{SL}_2(\mathbb{Z}))$ that satisfies Harder's congruence with $F$.
\end{prop}

\begin{proof}
We know that $\overline{\rho}_f$ is irreducible by the previous result. However by Theorem \ref{PUI} we have, under a suitable choice of basis: \[\overline{\rho}_{f,p} = \left(\begin{array}{cc}\overline{\lambda}_a & \star\\ 0 & \overline{\chi}_l \overline{\lambda}_a\end{array}\right)\quad\text{or}\quad \left(\begin{array}{cc}\overline{\chi}_l\overline{\lambda}_a & \star\\ 0 & \overline{\lambda}_a\end{array}\right).\]

In either case the restriction of $\overline{\rho}_{f,p}$ to the inertia subgroup $I_p$ of Gal$(\overline{\mathbb{Q}}_p/\mathbb{Q}_p)$ is as follows: \[\overline{\rho}_{f,p}\big|_{I_p} = \left(\begin{array}{cc}1 & \star'\\ 0 & 1\end{array}\right).\]

We have two cases. First it could be the case that $\star' \equiv 0 \bmod \Lambda$. If this is the case then we may use Ribet's level lowering theorem for modular representations (Theorem $1.1$ in \cite{ribet}) to produce $g\in S_{k'}(\text{SL}_2(\mathbb{Z}))$ such that $\overline{\rho}_{g} \sim \overline{\rho}_f$. We would then observe a level one version of Harder's congruence as required.

Now suppose that $\star' \not\equiv 0 \bmod \Lambda$, so that $\overline{\rho}_{f,p}$ is ramified. If $\pi_{F,p}$ is of type I or II$_b$ then $\pi_{F,p}$ is unramified. By the Local Langlands Correspondence for GSp$_4$ (proved in \cite{gan2}) we see that $\rho_{F,p}$ is unramified so that $\overline{\rho}_{F,p}$ is unramified, giving a contradiction. The only other possibility for $\pi_{F,p}$ is to be of type II$_a$ as required.
\end{proof}

To summarise our progress, given that $\pi_{F,p}$ is of type I or II then either: 
\begin{itemize}
\item{$f$ itself satisfies a simpler Ramanujan congruence, detected by a simple divisibility criterion,} 
\item{a replacement level $1$ elliptic form satisfies Harder's congruence with $F$. The level at which $F$ appears could be $1$ or $p$ since we did not place any ramification restrictions on $\pi_{F,p}$ in our assumptions,} 
\item{or $\pi_{F,p}$ is of type II$_a$, implying that a new paramodular form of level $p$ exists satisfying the congruence.} 
\end{itemize}

The first possibility is a rare occurrence and is easy to check for in practice. We will see later that the second possibility rarely occurs for $F\in S_{j,k}^{\text{new}}(K(p))$. The case where $F\in S_{j,k}(\text{Sp}_4(\mathbb{Z}))$ is of course the original Harder conjecture at level $1$.

From this discussion one should believe that the third possibility is most likely to occur if $F$ is not a lift from level $1$.

\subsection{Proving Theorem \ref{PQ}$(2)$}

Let us now assume that $\pi_{F,p}$ is induced from the Borel subgroup of GSp$_4(\mathbb{Q}_p)$. Then $\pi_{F,p}$ must be of type I-VI. We will show that if $p^{j+2k-2} \not\equiv 1 \bmod \Lambda$ for $t=0,1,2,3$ then $\pi_{F,p}$ is of type I or II.

Let $W_{\mathbb{Q}_p}' = \mathbb{C}\rtimes W_{\mathbb{Q}_p}$ be the Weil-Deligne group of $\mathbb{Q}_p$. The multiplication on this group is given by $(z,w)(z',w') = (z + \nu(w)z', ww')$, where $\nu: W_{\mathbb{Q}_p} \longrightarrow \mathbb{C}^{\times}$ is the character corresponding to $\mid\cdot\mid_p$ by local class field theory.

By the Local Langlands Correspondence for GSp$_4$ we may associate to each irreducible admissible representation $\pi$ of GSp$_4(\mathbb{Q}_p)$ its $L$-parameter, a certain representation: \[\rho_{\pi} : W_{\mathbb{Q}_p}' \longrightarrow \text{GSp}_4(\mathbb{C}).\] One can view such a representation as a pair $(\rho_0, N)$ where: \[\rho_0: W_{\mathbb{Q}_p} \longrightarrow \text{GSp}_4(\mathbb{C})\] is a continuous homomorphism and $N\in M_n(\mathbb{C})$ is a nilpotent matrix such that: \[\rho_0(w)N\rho_0(w)^{-1} = \nu(w)N,\] for all $w\in W_{\mathbb{Q}_p}$. Given $\rho_0$ and $N$ we recover the $L$-parameter via $\rho_{\pi}(z,w) = \rho_0(w)\text{exp}(zN)$.

Let $\pi_1,\pi_2$ be irreducible admissible representations of $\text{GSp}_4(\mathbb{Q}_p)$. Then $\pi_1 \cong \pi_2$ implies $\rho_{\pi_1} \cong \rho_{\pi_2}$ under the Local Langlands Correspondence. However the converse does not hold. A fixed $L$-parameter can arise from different isomorphism classes, but only finitely many (those in the same ``$L$-packet").

Roberts and Schmidt discuss the $L$-parameters of non-supercuspidal representations of GSp$_4(\mathbb{Q}_p)$ in \cite{schmidt2}. If $\pi$ is parabolically induced from the Borel subgroup of GSp$_4(\mathbb{Q}_p)$ then it is non-supercuspidal and $\rho_{\pi}$ is simple to describe. In particular the $\rho_0$ part is semisimple given by four characters $\mu_1,\mu_2,\mu_3,\mu_4$ of $W_{\mathbb{Q}_p}$ (which by local class field theory correspond to four characters of $\mathbb{Q}_p^{\times}$). See the Appendix for a complete table of $L$-parameters for Borel induced representations.

The four complex numbers $[\mu_1(p),\mu_2(p),\mu_3(p),\mu_4(p)]$ are the Satake parameters of $\pi$. Unramified representations are uniquely determined by their Satake parameters up to scaling (much in the same way as unramified local Galois representations are determined by the image of Frobenius).

Let us now return to our congruence between $f$ and $F$. We have already seen that the existence of this congruence for all $q\neq p$ leads to a residual equivalence of global Galois representations: \[\overline{\rho}_F \sim \overline{\rho}_f \oplus \overline{\chi}_l^{k-2} \oplus \overline{\chi}_l^{j+k-1}.\]

In particular we can compare these representations locally at $p$, the level of $f$. Since we have the local equality $\chi_{l}|_{W_{\mathbb{Q}_p}} = \nu^{-1}$ it follows that: \[\overline{\rho}_{F,p}|_{W_{\mathbb{Q}_p}} \sim \overline{\rho}_{f,p}|_{W_{\mathbb{Q}_p}} \oplus \overline{\nu}^{2-k} \oplus \overline{\nu}^{1-j-k}.\]

Given the existence of the congruence we see that the local representations $\rho_{F,p}|_{W_{\mathbb{Q}_p}}$ and $\rho_{f,p}|_{W_{\mathbb{Q}_p}}\oplus \nu^{2-k} \oplus \nu^{1-j-k}$ of $W_{\mathbb{Q}_p}$ have the same composition factors mod $\Lambda$. 

Recall that to $F$ we have attached a ``global" Galois representation $\rho_F$ and a ``global" automorphic representation $\pi_F$. Similarly for $f$. By local-global compatibility results (see \cite{sorensen2} for GSp$_4$ and \cite{taylor2} for GL$_2$) we know that $\rho_{F,p}|_{W_{\mathbb{Q}_p}}$ corresponds to $\pi_{F,p}$ and $\rho_{f,p}|_{W_{\mathbb{Q}_p}}$ corresponds to $\pi_{f,p}$ under the corresponding local Langlands correspondences. 

Tying all of this together, the existence of the congruence forces the $L$-parameter of $\pi_{F,p}$ to be congruent modulo $\Lambda$ to that of $\pi_{f,p}\oplus|\cdot|_p^{2-k}\oplus|\cdot|_p^{1-j-k}$ (up to scaling by $p^{\frac{k'-1}{2}} = p^{\frac{j+2k-3}{2}}$ in the first component). In particular the Satake parameters should match mod $\Lambda$.

Since $f$ is a newform of level $p$ it is known that $\pi_{f,p} \cong \text{St}$ or $\pi_{f,p} \cong \epsilon\text{St}$ where St is the Steinberg representation of GL$_2(\mathbb{Q}_p)$ and $\epsilon$ is the unique unramified non-trivial quadratic character of $\mathbb{Q}_p^{\times}$. In either case the Satake parameters are known to be $\alpha_p$ and $\alpha_p^{-1}$ where $\alpha_p = p^{\frac{1}{2}}$ or $\epsilon(p)p^{\frac{1}{2}} = -p^{\frac{1}{2}}$. Applying the scaling gives $[\alpha_p,\alpha_p^{-1}] = [p^{\frac{j+2k-2}{2}}, p^{\frac{j+2k-4}{2}}]$ or $[-p^{\frac{j+2k-2}{2}},-p^{-\frac{j+2k-4}{2}}]$.

It is now clear that the Satake parameters of $\pi_{f,p}\oplus|\cdot|_p^{2-k}\oplus|\cdot|_p^{1-j-k}$ are  \[[a,b,c,d] = \left[\pm p^{\frac{j+2k-2}{2}}, \pm p^{\frac{j+2k-4}{2}}, p^{k-2}, p^{j+k-1}\right]\] (where the sign is the same for $a$ and $b$). Note that these are all integral powers of $p$.

\begin{thm}\label{ABCD}
Suppose $\pi_{F,p}$ is of type I-VI and $p^{j+2t-2} \not\equiv 1 \bmod \Lambda$ for $t = 0,1,2,3$. Then $\pi_{F,p}$ cannot be of type III,IV,V or VI.
\end{thm}

\begin{proof}
Suppose $\pi_{F,p}$ is of one of the types III,IV,V,VI. We show that if the corresponding Satake parameters are congruent mod $\Lambda$ then $p^{j+2t-2} \equiv 1 \bmod \Lambda$ for some $t=0,1,2,3$. Then the result follows.

We work in reverse order. Here $\epsilon_0$ will stand for the trivial character. Whenever there is a choice of sign this will be fixed by a choice of upper or lower row.

\textbf{Type VI}
$\rho_0$ is given by the four characters \[\nu^{\frac{1}{2}}\sigma, \nu^{\frac{1}{2}}\sigma, \nu^{-\frac{1}{2}}\sigma, \nu^{-\frac{1}{2}}\sigma.\] Since the central character of $\pi_{F,p}$ is trivial we have $\sigma^2 = \epsilon_0$, so that $\sigma$ is trivial or quadratic.

Thus in some order the Satake parameters are given by \[\pm p^{\frac{1}{2}}, \pm p^{\frac{1}{2}}, \pm p^{-\frac{1}{2}}, \pm p^{-\frac{1}{2}}.\] Scaling by $p^{\frac{k'-1}{2}}$ gives \[\pm p^{\frac{j+2k-2}{2}}, \pm p^{\frac{j+2k-2}{2}}, \pm p^{\frac{j+2k-4}{2}}, \pm p^{\frac{j+2k-4}{2}}.\] Notice that there are two equal pairs here. Thus for $[a,b,c,d]$ to be congruent to these four numbers mod $\Lambda$ we would have to have that $a$ is equivalent to one of $b,c$ or $d$ mod $\Lambda$.

Setting $a\equiv b \bmod \Lambda$ gives $p \equiv 1 \bmod \Lambda$.

Setting $a \equiv c \bmod \Lambda$ gives $p^{\frac{j+2}{2}} \equiv \pm 1 \bmod \Lambda$.

Setting $a \equiv d \bmod \Lambda$ gives $p^{\frac{j}{2}} \equiv \pm 1 \bmod \Lambda$.

\textbf{Type V}
$\rho_0$ is given by the four characters \[\nu^{\frac{1}{2}}\sigma, \nu^{\frac{1}{2}}\xi\sigma, \nu^{-\frac{1}{2}}\xi\sigma, \nu^{-\frac{1}{2}}\sigma.\] Since the central character of $\pi_{F,p}$ is trivial we have $\sigma^2 = \epsilon_0$, so that $\sigma$ is trivial or quadratic.

Thus in some order the Satake parameters are given by \[\pm p^{\frac{1}{2}}, \mp p^{\frac{1}{2}}, \mp p^{-\frac{1}{2}}, \pm p^{-\frac{1}{2}}.\] Scaling by $p^{\frac{k'-1}{2}}$ gives \[\pm p^{\frac{j+2k-2}{2}}, \mp p^{\frac{j+2k-2}{2}}, \mp p^{\frac{j+2k-4}{2}}, \pm p^{\frac{j+2k-4}{2}}.\] Notice that there are two pairs of the form $(\alpha, -\alpha)$. Thus for $[a,b,c,d]$ to be congruent to these four numbers mod $\Lambda$ we would have to have that $a$ is equivalent to one of $-b,-c$ or $-d$ mod $\Lambda$.

Setting $a\equiv -b \bmod \Lambda$ gives $p \equiv 1 \bmod \Lambda$.

Setting $a \equiv -c \bmod \Lambda$ gives $p^{\frac{j+2}{2}} \equiv \mp 1 \bmod \Lambda$.

Setting $a \equiv -d \bmod \Lambda$ gives $p^{\frac{j}{2}} \equiv \mp 1 \bmod \Lambda$.

\textbf{Type IV}
$\rho_0$ is given by the four characters \[\nu^{\frac{3}{2}}\sigma, \nu^{\frac{1}{2}}\sigma, \nu^{-\frac{1}{2}}\sigma, \nu^{-\frac{3}{2}}\sigma.\] Since the central character of $\pi_{F,p}$ is trivial we have $\sigma^2 = \epsilon_0$, so that $\sigma$ is trivial or quadratic.

Thus in some order the Satake parameters are given by \[\pm p^{\frac{3}{2}}, \pm p^{\frac{1}{2}}, \pm p^{-\frac{1}{2}}, \pm p^{-\frac{3}{2}}.\] Scaling by $p^{\frac{k'-1}{2}}$ gives \[\pm p^{\frac{j+2k}{2}}, \pm p^{\frac{j+2k-2}{2}}, \pm p^{\frac{j+2k-4}{2}}, \pm p^{\frac{j+2k-6}{2}}.\] If $[a,b,c,d]$ are congruent to these numbers mod $\Lambda$ then there are four possibilities for $c$.

Setting $c \equiv \pm p^{\frac{j+2k}{2}} \bmod \Lambda$ gives $p^{\frac{j+4}{2}} \equiv \pm 1 \bmod \Lambda$.

Setting $c \equiv \pm p^{\frac{j+2k-2}{2}} \bmod \Lambda$ gives $p^{\frac{j+2}{2}} \equiv \pm 1 \bmod \Lambda$.

Setting $c \equiv \pm p^{\frac{j+2k-4}{2}} \bmod \Lambda$ gives $p^{\frac{j}{2}} \equiv \pm 1 \bmod \Lambda$.

Setting $c \equiv \pm p^{\frac{j+2k-6}{2}} \bmod \Lambda$ gives $p^{\frac{j-2}{2}} \equiv \pm 1 \bmod \Lambda$.

\textbf{Type III}
$\rho_0$ is given by the four characters \[\nu^{\frac{1}{2}}\chi\sigma, \nu^{-\frac{1}{2}}\chi\sigma, \nu^{\frac{1}{2}}\sigma, \nu^{-\frac{1}{2}}\sigma.\] Since the central character of $\pi_{F,p}$ is trivial we have $\chi\sigma^2 = \epsilon_0$, so that $\chi\sigma = \sigma^{-1}$.

Thus in some order the Satake parameters are given by \[p^{\frac{1}{2}}\beta^{-1}, p^{-\frac{1}{2}}\beta^{-1}, p^{\frac{1}{2}}\beta, p^{-\frac{1}{2}}\beta,\] where $\beta = \sigma(p)$. Scaling by $p^{\frac{k'-1}{2}}$ gives \[p^{\frac{j+2k-2}{2}}\beta^{-1}, p^{\frac{j+2k-4}{2}}\beta^{-1}, p^{\frac{j+2k-2}{2}}\beta, p^{\frac{j+2k-4}{2}}\beta.\] If $[a,b,c,d]$ are congruent to these numbers mod $\Lambda$ then there are four possibilities for $a$ (each giving the value of $\beta \bmod \Lambda$). However replacing $\beta$ by $\beta^{-1}$ gives the same Satake parameters, so it suffices to set $a$ congruent to just the last two Satake parameters.

Setting $a \equiv p^{\frac{j+2k-2}{2}}\beta \bmod \Lambda$ gives $\beta \equiv \pm 1 \bmod \Lambda$. This gives Satake parameters equivalent to \[\pm p^{\frac{j+2k-2}{2}}, \pm p^{\frac{j+2k-2}{2}}, \pm p^{\frac{j+2k-4}{2}}, \pm p^{\frac{j+2k-4}{2}}.\] However we have already dealt with these in Type VI.

Setting $a \equiv p^{\frac{j+2k-4}{2}}\beta \mod \Lambda$ gives $\beta \equiv \pm p \bmod \Lambda$. This gives Satake parameters equivalent to \[\pm p^{\frac{j+2k}{2}}, \pm p^{\frac{j+2k-2}{2}}, \pm p^{\frac{j+2k-4}{2}}, \pm p^{\frac{j+2k-6}{2}}.\] However we have already dealt with these in Type IV.

Suppose now that none of the following holds: \begin{align*}p^{j-2} &\equiv 1 \bmod \Lambda\\ p^j &\equiv 1 \bmod \Lambda\\ p^{j+2} &\equiv 1 \bmod \Lambda\\ p^{j+4} &\equiv 1 \bmod \Lambda.\end{align*} Then none of the conditions found above hold and so we must have that $\pi_{F,p}$ is of type I or II, as required.
\end{proof}

Note that if one compares the Satake parameters $[a,b,c,d]$ to those from a representation of type I or II then no conditions arise. It is always possible for these to be congruent mod $\Lambda$.

\subsection{Proving Theorem \ref{PQ}$(1)$}

We now move on to our final task, finding conditions that guarantee $\pi_{F,p}$ is induced from the Borel subgroup of GSp$_4(\mathbb{Q}_p)$. We will show that if $l\geq \text{max}\{6f(\Lambda)+2,e(\Lambda)+2\}$ then $\pi_{F,p}$ is induced from the Borel subgroup of $\text{GSp}_4$.

In this section $\Lambda'$ will be an arbitrary prime of $K = \mathbb{Q}_F\mathbb{Q}_f$, lying above a rational prime $l'$. 

Recall that $\pi_{F,p}$ corresponds via Local Langlands to a representation of the Weil-Deligne group $W_p'$, which itself is parametrized by a continuous representation $\rho_0: W_{\mathbb{Q}_p} \rightarrow \text{GSp}_4(\mathbb{C})$ and a nilpotent matrix $N\in\text{M}_4(\mathbb{C})$ with certain properties (mentioned in the previous subsection). However if we fix a choice of embeddings $\overline{\mathbb{Q}} \hookrightarrow\mathbb{C}$ and $\overline{\mathbb{Q}}\hookrightarrow\overline{\mathbb{Q}}_{l'}$ then one can convert these representations into $l'$-adic representations with open kernel (p.$77$ of \cite{taylor}).

It is also known that local Galois representations give rise to Weil-Deligne representations.

\begin{thm}{(Grothendieck-Deligne)}\label{PQRP}
Let $p\neq l'$ and fix a continuous $n$-dimensional $\Lambda'$-adic representation: \[\rho: \text{Gal}(\overline{\mathbb{Q}}_p/\mathbb{Q}_p) \longrightarrow \text{GL}_n(K_{\Lambda'}).\]

Then associated to $\rho$ is a unique $l'$-adic representation of $W_{\mathbb{Q}_p}'$, given by a pair $(\rho_0', N')$ satisfying:

\begin{itemize}
\item{$\rho_0': W_{\mathbb{Q}_p} \longrightarrow \text{GL}_n(K_{\Lambda'})$ is continuous with respect to the discrete topology on $\text{GL}_n(K_{\Lambda'})$. In particular $\rho_0'(I_p)$ is finite.}
\item{$\rho_0'(\phi_p)$ has characteristic polynomial defined over $\mathcal{O}_{\Lambda'}$ with constant term a $\Lambda'$-adic unit.}
\item{$N'\in\text{M}_n(K_{\Lambda'})$ is nilpotent and satisfies \[\rho_0'(\sigma)N'\rho_0'(\sigma)^{-1} = \nu(\sigma)N',\] for all $\sigma\in W_{\mathbb{Q}_p}$.}
\end{itemize}
Fixing a tamely ramified character $t_{l'}: I_p \rightarrow \mathbb{Z}_{l'}$, the relationship between $\rho$ and $\rho_0'$ is: \[\rho(\phi_p^n u) = \rho_0'(\phi_p^n u)\text{exp}(t_{l'}(u)N'),\] for all $n\in\mathbb{Z}$, $u\in I_p$.
\end{thm}

Now consider the local Galois representation $\rho_{F,p}$. By the above theorem it has an associated Weil-Deligne representation, given by a pair $(\rho_0',N')$. A Local-Global Compatibility conjecture of Sorensen (pages $3$-$4$ of \cite{sorensen2}, proved in certain cases by Mok in Theorem $4.14$ of \cite{mok}) shows that the Weil-Deligne representations attached to $\pi_{F,p}$ and $\rho_{F,p}$ are isomorphic (up to Frobenius semi-simplification). In particular this implies that $\rho_0 \cong \rho_0'$ up to semi-simplification. We make this identification from now on and use $\rho_0$ to denote the Frobenius semi-simplification of both representations.

A useful corollary of the above theorem is the following:

\begin{cor}{(Grothendieck Monodromy Theorem)}
With the above setup there exists a finite index subgroup $J_{\Lambda'}\subseteq I_p$ such that $\rho(\sigma) = \text{exp}(t_{l'}(\sigma)N)$ for each $\sigma\in J_{\Lambda'}$, i.e. each element of $J_{\Lambda'}$ acts unipotently.
\end{cor}

See the appendix of \cite{serre} for a proof of this.

By the Grothendieck Monodromy Theorem there exists a (maximal) finite index subgroup $J_{\Lambda'}\subseteq I_p$ acting by unipotent matrices, i.e. if $\sigma\in J_{\Lambda'}$ then: \[\rho_{F,p}(\sigma) = \text{exp}(t_{l'}(\sigma)N).\] Note then that as a consequence, for each $\sigma\in J_{\Lambda'}$: \[\rho_0(\sigma) = \rho_{F,p}(\sigma)\text{exp}(-t_{l'}(\sigma)N) = I.\] Thus $\rho_0$ factors through $I_p/J_{\Lambda'}$: \[\rho_0: I_p \longrightarrow I_p/J_{\Lambda'} \longrightarrow \text{GL}_4(\mathcal{O}_{\Lambda'}).\]

Note that $\rho_0(I_p/J_{\Lambda'})$ is finite. It is conjectured that the size of this image is independent of $\Lambda'$ (see Conjecture $1.3$ of \cite{taylor}).

We wish to show that $\pi_{F,p}$ is induced from the Borel subgroup of $\text{GSp}_4$. It suffices to show that $J_{\Lambda'} = I_p$ for some $\Lambda'$ (this case is commonly known as ``semi-stable").

\begin{prop}
If $\Lambda'$ satisfies $J_{\Lambda'} = I_p$ then $\pi_{F,p}$ is induced from the Borel subgroup of GSp$(\mathbb{Q}_p)$. 
\end{prop}

\begin{proof}
It suffices to show that there is a basis of $K_{\Lambda'}^4$ such that \[\rho_0 \cong \left(\begin{array}{cccc}\chi_1 & \star & \star & \star\\ 0 & \chi_2 & \star & \star \\ 0 & 0 & \chi_3 & \star\\ 0 & 0 & 0 & \chi_4\end{array}\right),\] for four unramified characters $\chi_1, \chi_2, \chi_3, \chi_4$ of $W_{\mathbb{Q}_p}$. Then since the image of $\rho_0$ lies in $\text{GSp}_4$ we must have that $\chi_3 = \chi_1^{-1}$ and $\chi_4 = \chi_2^{-1}$. Then by Local Langlands for GSp$_4$ it must be that $\pi_{F,p}$ is induced from the Borel subgroup.

To this end we already know that $I_p$ acts unipotently and so it remains to study the action of Frobenius $\phi_p$. Recall the condition $\rho_0(\phi_p) N \rho_0(\phi_p)^{-1} = p^{-1} N$. We will rewrite this as $\rho_0(\phi_p) N = p^{-1} N \rho_0(\phi_p)$.

By Theorem \ref{PQRP} the characteristic polynomial of $\rho_0(\phi_p)$ has constant term in $\mathcal{O}_{\Lambda'}^{\times}$. Choosing an eigenvector $v$ of $\rho_0(\phi_p)$ with non-zero eigenvalue $\alpha\in\mathcal{O}_{\Lambda'}$, notice that \[\rho_0(\phi_p)(Nv) = p^{-1} N \rho_0(\phi_p) = \alpha p^{-1} (Nv).\] This shows that if $Nv \neq 0$ then $Nv$ is another eigenvector of $\rho_0(\phi_p)$ with eigenvalue $\alpha p^{-1} \neq \alpha$.

Consider the list $v, Nv, N^2v, N^3v$. If all of these vectors are non-zero then we have a basis of eigenvectors for $\rho_0(\phi_p)$. Then $\rho_0(\phi_p)$ is diagonal.

If for some $i\leq 3$ we have $N^i v = 0$ then we can quotient out by the subspace generated by $v,Nv, ..., N^{i-1}v$ and apply the same argument to the quotient, lifting basis vectors to $K_{\Lambda'}^4$ where necessary.

Continuing in this fashion we then construct a basis of $K_{\Lambda'}^4$ such that:

\[\rho_0(\phi_p) = \left(\begin{array}{cccc}\alpha_1 & \star & \star & \star\\ 0 & \alpha_2 & \star & \star\\ 0 & 0 & \alpha_3 & \star\\ 0 & 0 & 0 & \alpha_4\end{array}\right).\] 

It is then clear that $\rho_0$ is of the required form with unramified characters defined by $\chi_i(\phi_p) = \alpha_i$ for $i=1,2,3,4$ (since $I_p$ acts unipotently).
\end{proof}

If $J_{\Lambda} = I_p$ then the Proposition shows that $\pi_{F,p}$ is induced from the Borel. It is our aim to find a condition guaranteeing this. First we study the possible sizes of $\rho_0(I_p/J_{\Lambda'})$.

\begin{lem}\label{HY}
Suppose $G$ is a finite subgroup of $\text{GL}_n(\mathcal{O}_{\Lambda'})$ and that $l' > e(\Lambda')+1$ (where $e(\Lambda')$ is the ramification index of $K_{\Lambda'}/\mathbb{Q}_{l'}$). Then the reduction map injects $G$ into $\text{GL}_n(\mathbb{F}_{\Lambda'})$.
\end{lem}

\begin{proof}
We show that the reduction map $\text{GL}_n(\mathcal{O}_{\Lambda'}) \rightarrow \text{GL}_n(\mathbb{F}_{\Lambda'})$ is torsion-free. Then the restriction of this map to $G$ must have trivial kernel, so that $G$ injects into GL$_n(\mathbb{F}_{\Lambda'})$.

To prove the claim we take $A\in\text{GL}_n(\mathcal{O}_{\Lambda'})$ with $A\neq I$ and $A \equiv I \bmod \Lambda'$. We wish to prove that $A^m \neq I$ for each $m$. We already know this for $m=1$.

Suppose that $A$ has finite order $m>1$. Choose a prime $q\mid m$, so that $m=qk$ for some $k\geq 1$. Letting $B = A^k$ we see that $B^q = I$, $B \neq I$ and that $B \equiv I \bmod \Lambda'$. Thus we can assume without loss of generality that $A$ has prime order $q$.

To this end we write $A = I + M$ with $M\neq 0$ and $M$ having entries in $\Lambda'$. Choose an entry $m_{u,v}$ of $M$ such that $|m_{u,v}|_{\Lambda'} = \delta$ is maximal among all entries of $M$. Then $0<\delta\leq \frac{1}{N(\Lambda')}$.

Note that: \[A^q = (I+M)^q = I + qM + \binom{q}{2}M^2 + ... + \binom{q}{q-1}M^{q-1} + M^q.\]

\textbf{Case 1:} Suppose $q\neq l'$. Then the entries of $\binom{q}{j}M^j$ for $j\geq 2$ all have $\Lambda'$-adic absolute value less than or equal to $\delta^2$. However $qM$ contains the entry $qm_{u,v}$ of absolute value $\delta > \delta^2$ (since $q\neq l'$). Hence $A^q-I$ must contain an entry of absolute value $\delta>0$ and so $A^q - I \neq 0$ as required.

\textbf{Case 2:} $q=l'$. We need sharper inequalities for this case since $qM$ has no entry of absolute value $\delta$. It is now the case that $qm_{u,v}$ has absolute value $\frac{\delta}{N(\Lambda')^e}$ (where $e = e(\Lambda')$). 

For $2\leq j\leq q-1$ we know that $q$ divides $\binom{q}{j}$ so the matrices $\binom{q}{j}M^j$ have entries of maximal absolute value $\frac{\delta^2}{N(\Lambda')^e} <\frac{\delta}{N(\Lambda')^e}$.  Also the matrix $M^q$ has entries of absolute value greater than or equal to $\delta^q < \delta^{e+1} \leq \frac{\delta}{N(\Lambda')^e}$ (using here the condition $q = l' > e+1$).

Thus we see that $A^q - I$ contains an entry of absolute value $\frac{\delta}{N(\Lambda')^e}>0$ and so $A^q - I \neq 0$ as required. 
\end{proof}

We now know that $m_{\Lambda'} = |\rho_0(I_p/J_{\Lambda'})|$ divides $|\text{GL}_4(\mathbb{F}_{\Lambda'})|$ whenever $l'>e(\Lambda')+1$. By using the existence of the congruence we may find a further restriction on the size of this image.

\begin{lem}
If $l > e(\Lambda)+1$ and $J_{\Lambda} \neq I_p$ then $N(\Lambda)|m_{\Lambda'}$ for all $\Lambda'$.
\end{lem}

\begin{proof}
Note that $m_{\Lambda'} = |\rho_{F,p}(I_p/J_{\Lambda'})|$. As mentioned earlier it is conjectured that $m_{\Lambda'}$ has order independent of $\Lambda'$. Thus it suffices to show that $N(\Lambda)|m_{\Lambda}$.

Now $G = \rho_{F,p}(I_p/J_{\Lambda})$ is a finite subgroup of GL$_4(\mathcal{O}_{\Lambda})$ and $l>e(\Lambda)+1$. By Lemma \ref{HY} the reduction map injects $G$ into $\text{GL}_4(\mathbb{F}_{\Lambda})$. Thus $|G| = |\overline{\rho}_{F,p}(I_p/J_{\Lambda'})|$.

However by the existence of the congruence we know that the mod $\Lambda$ reduction $\overline{\rho}_{F,p}$ has composition factors $\overline{\rho}_{f,p}, \overline{\chi}_l^{k-2}, \overline{\chi}_l^{j+k-1}$. 

Then since $\overline{\rho}_{f,p}|_{I_p} = \left(\begin{array}{cc}1 & \star\\ 0 & 1\end{array}\right)$ and $\chi_l$ is unramified at $p$ we have: \[\overline{\rho}_{F,p}(I_p/J_{\Lambda}) \subseteq \left\{\left(\begin{array}{cccc}1 & \star & \star & \star\\ 0 & 1 & \star & \star\\ 0 & 0 & 1 & \star\\ 0 & 0 & 0 & 1\end{array}\right)\right\}.\] However $J_{\Lambda} \neq I_p$ by assumption, so that $\overline{\rho}_{F,p}(I_p/J_{\Lambda})$ is non-trivial. This shows that $N(\Lambda)$ divides $m_{\Lambda}$, as required.
\end{proof}

\begin{cor}
If $l \geq \text{max}\{6f(\Lambda)+2,e(\Lambda)+2\}$ then $J_{\Lambda} = I_p$ (here $f(\Lambda)$ is the residue degree of $K_{\Lambda}/\mathbb{Q}_l$).
\end{cor}

\begin{proof}
Suppose $J_{\Lambda} \neq I_p$. Then we know that $N(\Lambda)|m_{\Lambda'}$ for all $\Lambda'$. But for each $\Lambda'$ satisfying $l' > e(\Lambda')+1$ we know that $m_{\Lambda'}$ divides $|\text{GL}_4(\mathbb{F}_{\Lambda'})|$ and so (writing $f = f(\Lambda)$) \[l| l'^{6f}(l'^f-1)(l'^{2f}-1)(l'^{3f}-1)(l'^{4f}-1).\]

It remains to prove that $l'$ can be chosen to contradict this. To contradict the divisibility condition it suffices to choose $l' \neq l$ such that $l'^{3f} \not\equiv 1 \bmod l$ and $l'^{4f} \not\equiv 1 \bmod l$. 

There are at most $3f+4f=7f$ classes mod $l$ that have order dividing $3f$ or $4f$. However note that the classes of order dividing $\text{hcf}(3f,4f) = f$ are counted twice and so there must be at most $7f-f = 6f$ classes of order dividing $3f$ and $4f$. But since $l \geq 6f+2$ there must be a non-zero class $a$ mod $l$ that has order coprime to $3f$ and $4f$. By Dirichlet's theorem there are infinitely many primes in this class mod $l$. It suffices to choose $l' \equiv a \bmod l$ such that $l' > e(\Lambda')+1$.
\end{proof}

Of course it is highly likely that $l \geq \text{max}\{6f(\Lambda)+2,e(\Lambda)+2\}$ in practice since $l$ is a ``large" prime.

\section{Congruences of local origin}
One can use similar techniques to Subsection $3.2$ to explain why congruences of ``local origin" are rare for new paramodular forms. To explain the theory of such congruences we first survey the known results for elliptic modular forms. For a more in depth discussion see \cite{danneil}.

Recall that, for all primes $p$ we have the Ramanujan congruence: \[\tau(p) \equiv 1 + p^{11} \bmod 691.\] This shows a congruence between Hecke eigenvalues of a level $1$ cuspform of weight $12$ and the Hecke eigenvalues of the weight $12$ Eisenstein series. 

The modulus $691$ can be interpreted in many ways. Naively this prime just happens to appear in the $q$-expansion of $E_{12}$. A better interpretation is that it divides the numerator of $\frac{B_{12}}{24}$ (the relevant quantity in the coefficients of $E_{12}$). However the best interpretation is that it divides the numerator of $\frac{\zeta(12)}{\pi^{12}}$.

Ramanujan's congruence can be extended to give other Eisenstein congruences for even weights $k\geq 12$. The following is proved in \cite{datskovsky}.

\begin{thm} 
Suppose $\text{ord}_l\left(\frac{\zeta(k)}{\pi^k}\right) > 0$ for some prime $l$. Then there exists a normalised eigenform $f\in S_k(\text{SL}_2(\mathbb{Z}))$ with eigenvalues $a_n$ such that: \[a_p \equiv 1 + p^{k-1} \bmod \Lambda,\] for all primes $p$ (here $\Lambda\mid l$ in $\mathbb{Q}_f$).
\end{thm}

One can ask whether such Eisenstein congruences arise for elliptic modular forms of higher level. Indeed they do. Consider the question of finding a normalized eigenform $f\in S_k(\Gamma_0(p))$ satisfying for all $q\neq p$: \[a_q \equiv 1 + q^{k-1} \bmod \lambda\] where $\lambda$ is some prime of $\mathbb{Q}_f$. For technical reasons we must demand that $k\neq 2$ and that $\lambda$ does not lie above $2$ or $3$.

Of course if ord$_{\lambda}\left(\frac{\zeta(k)}{\pi^k}\right) > 0$ then the above theorem provides a level $1$ cuspform that satisfies the congruence (i.e. an oldform in $S_k(\Gamma_0(p))$). However newforms can satisfy such congruences too. How do we account for these?

It turns out that instead of looking for primes dividing (global) zeta values we can instead look for primes dividing incomplete zeta values. Let: \[\zeta_{\{p\}}(s) = \prod_{q\neq p}\left(1 - \frac{1}{q^s}\right)^{-1} = \left(1 - \frac{1}{p^s}\right)\zeta(s) = \frac{(p^k-1)}{p^s}\zeta(s).\] The following is proved in \cite{danneil}.

\begin{thm}
Let $p$ be prime and $k\geq 4$ be even. Suppose $l>3$ satisfies $\text{ord}_l\left(\frac{\zeta_{\{p\}}(k)}{\pi^k}\right) > 0$. Then there exists $f\in S_k(\Gamma_0(p))$, a normalized Hecke eigenform away from $p$ with eigenvalues $a_q\in\mathcal{O}_f$ satisfying: \[a_q \equiv 1 + q^{k-1} \bmod \Lambda,\]  for all $q\neq p$ and for some $\Lambda\mid l$ in $\mathbb{Q}_f$.
\end{thm}

Notice $\frac{\zeta_{\{p\}}(k)}{\pi^k} \sim \frac{B_{k}(p^k-1)}{2k}$, a condition we saw in Theorem \ref{PQ}. The term ``local origin" is used to describe the new congruences arising from divisibility of a (local) Euler factor.  As mentioned above, the local origin congruences generally come from newforms (since divisibility of the zeta value gives a congruence at level $1$).

We can do this in more generality. Let $\Sigma$ be a finite set of primes and set \[\zeta_{\Sigma}(s) = \prod_{p\notin \Sigma}\left(1 - \frac{1}{p^s}\right)^{-1} = \prod_{p\in\Sigma}\left(1 - \frac{1}{p^s}\right)\zeta(s) = \prod_{p\in\Sigma}\frac{(p^k - 1)}{p^s}\zeta(s).\] Then one can predict similar congruences for higher level newforms coming from divisibility of the special values $\frac{\zeta_{\Sigma}(k)}{\pi^k}$.

Naturally we may ask whether ``local origin" analogues of Conjecture \ref{V} exist. Indeed these are also predicted to occur and a plentiful supply of evidence has been found \cite{bergstrom2}. However, unlike the $\text{GL}_2$ case, these congruences are still conjectural.

Given a normalized Hecke eigenform $f\in S_{k'}(\text{SL}_2(\mathbb{Z}))$ with eigenvalues $a_q\in\mathcal{O}_f$ and a fixed prime $p$ we define an incomplete L-function of $f$: \[L_{\{p\}}(f,s) = (1 - a_p p^{-s} + p^{k'-1-2s})L(f,s) = \frac{(p^{2s} - a_p p^s + p^{k'-1})}{p^{2s}}L(f,s).\] The following congruences are then predicted by Harder in \cite{harder2}.

\begin{conj}
Let $j>0$ and $k\geq 3$. Let $f\in S_{j+2k-2}(\text{SL}_2(\mathbb{Z}))$ be a normalized Hecke eigenform with eigenvalues $a_p\in\mathcal{O}_f$. Suppose $\text{ord}_{\lambda}\left(\frac{L_{\{p\}}(f,j+k)}{\Omega_{j+k}}\right) > 0$ for some prime $\lambda$ in $\mathbb{Q}_f$. Then there exists $F\in S_{j,k}(\Gamma_0(p))$, a Hecke eigenform away from $p$ with eigenvalues $b_q\in\mathcal{O}_F$ satisfying \[b_q \equiv q^{k-2} + a_q + q^{j+k-1} \bmod \Lambda\] for all $q\neq p$ and for some $\Lambda\mid \lambda$ in $\mathbb{Q}_f\mathbb{Q}_F$.
\end{conj}

As in the $\text{GL}_2$ case, congruences arising from divisibility of the Euler factor are described as local origin congruences. It is also expected, in analogy, that local origin congruences generally come from newforms (since divisibility of the L-value would give a congruence at level $1$, by the original Harder conjecture).

One may ask whether it is possible to find local origin congruences for paramodular newforms. We will see below that these are surprisingly very rare. To this end suppose an eigenform $F\in S_{j,k}^{\text{new}}(K(p))$ away from $p$ satisfies a local origin congruence with a normalized eigenform $f\in S_{k'}(\text{SL}_2(\mathbb{Z}))$ and modulus $\Lambda\mid \lambda$ in $\mathbb{Q}_f\mathbb{Q}_F$. 

Recall that, by discussions in Subsection $3.2$, the existence of the congruence forces $\pi_{F,p}$ to have Satake parameters congruent to $\alpha_p, \alpha_p^{-1}, p^{j+k-1}, p^{k-2}$ mod $\lambda$. However now that $f$ is of level $1$ the values of $\alpha_p, \alpha_p^{-1}$ are different.

Fortunately we only need to know these values mod $\Lambda$ and the divisibility of the Euler factor at $p$ gives this. Indeed: \[\Lambda\mid (p^{2(j+k)} - a_p p^{j+k} + p^{k'-1}) = (p^{j+k} + \alpha_p p^{\frac{k'-1}{2}})(p^{j+k} + \alpha_p^{-1} p^{\frac{k'-1}{2}})\] and so $\alpha_p \equiv p^{\pm\left(\frac{j+3}{2}\right)} \bmod \Lambda$. Scaling by $p^{\frac{j+2k-3}{2}}$ gives Satake parameters $[p^{j+k},p^{k-3}]$ for $\pi_{f,p}$.

So, assuming the existence of a local origin congruence the Satake parameters $[a,b,c,d]$ of $\pi_{F,p}$ must be congruent to $[p^{j+k},p^{k-3}, p^{j+k-1}, p^{k-2}]$ in some order.

\begin{thm}
If a local origin congruence occurs for $F\in S_{j,k}^{\text{new}}(K(p))$ with modulus $\Lambda$ then $p^{j+2t} \equiv 1 \bmod \Lambda$ for some $t=0,1,2,3$.
\end{thm}

\begin{proof}
Since $F\in S_{j,k}^{\text{new}}(K(p))$ we know that $\pi_{F,p}$ is of type II$_a$, IV$_c$, V$_b$, V$_c$ or VI$_c$ (these are the only types with new $K(p)$ fixed vectors).

Comparing Satake parameters as in Theorem \ref{ABCD} gives the result. The details are omitted.
\end{proof}

\newpage
\appendix
\section{Borel induced representations of GSp$_4$}

The following table, extracted from p.$297$ of Roberts and Schmidt \cite{schmidt2}, lists the classification of all non-supercuspidal irreducible admissible representations of GSp$_4(\mathbb{Q}_p)$ induced from the Borel subgroup. For simplicity only the induced representations are given, rather than their irreducible constituents.

Contained in the table is information about dim$(V^K)$ for certain interesting open compact subgroups $K$ of GSp$_4(\mathbb{Q}_p)$ (i.e. GSp$_4(\mathbb{Z}_p)$ and the local paramodular group $K(p)$).

Here $\epsilon_0$ is the trivial character and $\xi$ is the unique unramified quadratic character of $\mathbb{Q}_p$. Also $\chi_1, \chi_2, \sigma$ are unramified characters.

\begin{center}
\begin{tabular}{|c|c|c|c|c|}
\hline Type & Constituent of & Conditions & dim$(V^{\text{GSp}_4(\mathbb{Z}_p)})$ & dim$(V^{K(p)})$\\
\hline
I & $\chi_1\times \chi_2 \rtimes \sigma$ & $\chi_1, \chi_2 \neq |\cdot|_p^{\pm 1}, \chi_1 \neq |\cdot|_p^{\pm 1} \chi_2^{\pm 1}$ & $1$ & $2$\\
\hline
II$_a$ & $|\cdot|_p^{\frac{1}{2}}\chi\times |\cdot|_p^{-\frac{1}{2}}\chi\rtimes \sigma$  & $\chi \neq |\cdot|_p^{\pm\frac{3}{2}}, \chi^2 \neq |\cdot|_p^{\pm 1}$ & $0$ & $1$\\
II$_b$ &  & & $1$ & $1$\\
\hline
III$_a$ & $\chi\times |\cdot|_p\rtimes |\cdot|_p^{-\frac{1}{2}}\sigma$  & $\chi\neq \epsilon_0, |\cdot|_p^{\pm 2}$ & $0$ & $0$\\
III$_b$ &  & & $1$ & $2$\\
\hline
IV$_a$ & $|\cdot|_p \times |\cdot|_p^2\rtimes |\cdot|_p^{-\frac{3}{2}}\sigma$  & & $0$ &$0$\\
IV$_b$ &  & &$0$ &$0$\\
IV$_c$ &  & &$0$ &$1$\\
IV$_d$ &  & &$1$ &$1$\\
\hline
V$_a$ & $|\cdot|_p\xi\times\xi\rtimes |\cdot|_p^{-\frac{1}{2}}\sigma$  & $\xi^2 = \epsilon_0, \xi\neq \epsilon_0$ &$0$ &$0$\\
V$_b$ &  & &$0$ &$1$\\
V$_c$ &  & &$0$ &$1$\\
V$_d$ &  & &$1$ &$0$\\
\hline
VI$_a$ &$|\cdot|_p\times\epsilon_0\rtimes |\cdot|_p^{-\frac{1}{2}}\sigma$  & &$0$ &$0$\\
VI$_b$ &  & &$0$ &$0$\\
VI$_c$ &  & &$0$ &$1$\\
VI$_d$ &  & &$1$ &$1$\\
\hline
\end{tabular}
\end{center}

The following table, extracted from p.$283$ of Roberts and Schmidt \cite{schmidt2}, gives the corresponding $L$-parameters. The matrices $N$ will not be needed so have been omitted.

\begin{center}
\begin{tabular}{|c|c|c|c|}
\hline Type & $\rho_0$ & Central character \\
\hline
I & $\chi_1\chi_2\sigma, \chi_1\sigma, \chi_2\sigma, \sigma$ & $\chi_1\chi_2\sigma^2$\\
\hline
II &$\chi^2\sigma, \nu^{\frac{1}{2}}\chi\sigma, \nu^{-\frac{1}{2}}\chi\sigma, \sigma$ & $(\chi\sigma)^2$\\
\hline
III&$\nu^{\frac{1}{2}}\chi\sigma, \nu^{-\frac{1}{2}}\chi\sigma, \nu^{\frac{1}{2}}\sigma, \nu^{-\frac{1}{2}}\sigma$ & $\chi\sigma^2$\\
\hline
IV & $\nu^{\frac{3}{2}}\sigma, \nu^{\frac{1}{2}}\sigma, \nu^{-\frac{1}{2}}\sigma, \nu^{-\frac{3}{2}}\sigma$ &$\sigma^2$\\
\hline
V& $\nu^{\frac{1}{2}}\sigma, \nu^{\frac{1}{2}}\xi\sigma, \nu^{-\frac{1}{2}}\xi\sigma, \nu^{-\frac{1}{2}}\sigma$ &$\sigma^2$\\
\hline
VI & $\nu^{\frac{1}{2}}\sigma, \nu^{\frac{1}{2}}\sigma, \nu^{-\frac{1}{2}}\sigma, \nu^{-\frac{1}{2}}\sigma$ &$\sigma^2$\\
\hline
\end{tabular}
\end{center}


\newpage
\begin{thebibliography}{99}

\bibitem{asgari}
M. Asgari, R. Schmidt,
\emph{Siegel modular forms and representations}.
Manuscripta math. $104$, $173-200$, 
$2001$.

\bibitem{bergstrom}
J. Bergstr\"om, C. Faber, G. van der Geer,
\emph{Siegel Modular Forms of Genus $2$ and Level $2$: Cohomological Computations and Conjectures}.
Int. Math. Res. Not. IMRN, 
$2008$.

\bibitem{bergstrom2}
J. Bergstr\"om, N. Dummigan,
\emph{Eisenstein congruences for split reductive groups}.
$2014$.

\url{http://people.su.se/~jonab/eiscong1.pdf}.

\bibitem{bernstein}
J. Bernstein, S. Gelbart at al,
\emph{An introduction to the Langlands Program}.
Birkhauser Basel, 
$2004$.

\bibitem{brown}
J. Brown
\emph{Congruence primes of Saito-Kurokawa lifts}
Mathematical Research Letters, $17(5), 977-991$, 
$2010$.

\bibitem{123}
J. H. Bruinier, G. van der Geer, G. Harder, D. Zagier,
\emph{The $1$-$2$-$3$ of Modular Forms}.
Springer,
Universitext,
$2008$.

\bibitem{cassels}
J. W. S. Cassels,
\emph{Local Fields}.
Cambridge University Press, London Mathematical Society Student Texts	(No. $3$),
$1986$.

\bibitem{chenevier2}
G. Chenevier, J. Lannes,
\emph{Formes automorphes et voisins de Kneser des réseaux de Niemeier}.
\url{http://arxiv.org/pdf/1409.7616v2.pdf}.

\bibitem{childress}
N. Childress,
\emph{Class Field Theory}.
Springer, Universitext,
$2008$.

\bibitem{datskovsky}
B. Datskovsky, P. Guerzhoy,
\emph{On Ramanujan Congruences for Modular Forms of Integral and Half-Integral Weights}
Proceedings of the American Mathematical Society, 
Volume $124$, No. $8$, Pages $2283-2291$,
$1996$. 

\bibitem{deligne1}
P. Deligne,
\emph{Values of L-Functions and Periods of Integrals}.
Proc. Symp. Pure Math. 
AMS $33$, $313-346$,
$1979$.

\bibitem{danneil}
N. Dummigan, D. Fretwell,
\emph{Ramanujan style congruences of local origin}.
Journal of Number Theory,
Volume $143$, Pages $248-261$,
$2014$.

\bibitem{edixhoven}
B. Edixhoven,
\emph{The weight in Serre's conjectures on modular forms}.
Inventiones mathematicae, Volume $109$, Issue $1$, $563-594$,
$1992$. 

\bibitem{fretwell}
D. Fretwell,
\emph{Level $p$ paramodular congruences of Harder type}.
 arXiv:1603.07088 [math.NT],
$2016$.

\bibitem{gan2}
W. T. Gan, S. Takeda,
\emph{The Local Langlands Conjecture for GSp$(4)$}.
Annals of Mathematics, Pages $1841-1882$, Volume $173$, Issue $3$,
$2011$. 

\bibitem{harder1}
G. Harder,
\emph{A congruence between a Siegel and an Elliptic Modular Form}.
Featured in ``The $1$-$2$-$3$ of Modular Forms". 

\bibitem{harder2}
G. Harder,
\emph{Secondary Operations in the Cohomology of Harish-Chandra modules},
$2013$.

\url{http://www.math.uni-bonn.de/people/harder/Manuscripts/Eisenstein/SecOPs.pdf}.

\bibitem{hida}
H. Hida,
\emph{Geometric Modular Forms and Elliptic Curves}.
World Scientific publishing,
$2000$.

\bibitem{jarvis}
F. Jarvis,
\emph{Level lowering for modular mod $l$ representations over totally real fields}.
Mathematische Annalen $12; 313(1):141-160$,
$1998$. 

\bibitem{lehmer}
 D. H. Lehmer, 
\emph{The vanishing of Ramanujan’s function $\tau(n)$}. 
Duke Math. J. $14$: $429–433$,
$1947$.

\bibitem{mok}
C. P. Mok,
\emph{Galois representations attached to automorphic forms on GL2 over CM fields}.
Compositio Math. 150, 523-567,
$2014$.

\bibitem{poor}
C. Poor, D. Yuen,
\emph{Paramodular cusp forms}.
Journal Math. Comp. $84$, $1401-1438$,
$2015$. 

\bibitem{ribet}
K. Ribet,
\emph{On Modular Representations of Gal(Q/Q) arising from modular forms}. 
Invent. Math. $100, 431–476$,
$1990$.

\bibitem{schmidt2}
B. Roberts, R. Schmidt,
\emph{Local Newforms for GSp$_4$}.
Springer, Lecture notes in mathematics $1918$,
$2007$.

\bibitem{schmidt1}
B. Roberts, R. Schmidt,
\emph{On modular forms for the paramodular group}. 
Automorphic Forms and Zeta Functions. 
Proceedings of the Conference in Memory of Tsuneo Arakawa. 
World Scientific, 
$2006$. 

\bibitem{serre}
J. P. Serre, J. Tate,
\emph{Good reduction of Abelian varieties}.
Annals of Mathematics, Vol $88$, $3$, $492-517$,
$1968$.

\bibitem{sorensen2}
C. M. Sorensen,
\emph{Galois representations and Hilbert-Siegel modular forms}. 
Doc. Math. $15$, $623-670$,
$2010$. 

\bibitem{taylor}
R. Taylor,
\emph{Galois Representations}.

\url{http://www.math.ias.edu/~rtaylor/longicm02.pdf}.

\bibitem{taylor2}
R. Taylor, T. Yoshida,
\emph{Compatibility of local and global Langlands correspondences}. 
J. Amer. Math. Soc., $20-2, 467-493$, 
$2007$.

\bibitem{weissauer}
R. Weissauer,
\emph{Four dimensional Galois representations}.
$2000$.

\bibitem{weise}
G. Wiese
\emph{Galois Representations}

\url{http://math.uni.lu/~wiese/notes/GalRep.pdf}

\end{thebibliography}
\end{document}